\documentclass[11pt]{article}
\usepackage{amsmath,enumerate}
\usepackage{amssymb,amsthm}
\usepackage{latexsym}
\usepackage[latin1]{inputenc}
\usepackage{epsfig}
\usepackage{color}
\newcommand{\PP}{{\mathbb{P}}}

\newcommand{\indicator}[1]{1_{\left\{{#1}\right\}}}
\newcommand{\de}{\text{\normalfont{d}}}
\newcommand{\e}{\text{\normalfont{e}}}
\newcommand{\F}{\overline{F}}
\newcommand{\G}{\overline{G}}

\newcommand{\laplace}[1]{\hat L_{#1}}
\newcommand{\I} {{\iota}}

\newcommand{\mean}[1]{\mathbb E\left[ #1\right]}

\usepackage{srcltx}

\newcommand{\grossO} {\mathcal O}

\allowdisplaybreaks[1]
\newtheorem{theorem}{Theorem}[section]

\newtheorem{lemma}[theorem]{Lemma}

\newtheorem{assum}{Assumption}[section]

\newtheorem{rem}{Remark}[section]

 \usepackage{epsfig}
\usepackage{amssymb}
\usepackage{amsmath}

\newcommand{\Ree}{\text{\normalfont{Re}}}
\newcommand{\Ime}{\text{\normalfont{Im}}}

\title{\sc\large Second order corrections for the limits of normalized ruin times
in the presence of heavy tails}

\author{ By {\sc S\o ren Asmussen$~^a$ and Dominik Kortschak$~^{b}$\footnote{This work was partially supported by a grant from the Thiele Center at Aarhus University, by the
the MIRACCLE-GICC project and the Chaire d'excellence "Generali - Actuariat responsable: gestion des risques naturels et changements climatiques."}}}
\date{\scriptsize\em
$~^a$ Department of Mathematical Sciences, Aarhus University, Ny Munkegade, DK-8000 Aarhus C, Denmark,\\[0.1cm]
$~^b$Universit\'e de Lyon, F-69622, Lyon, France; Universit\'e Lyon 1, Laboratoire SAF, EA 2429, Institut de Science Financi\`ere et d'Assurances, 50 Avenue Tony Garnier, F-69007 Lyon, France
}

\begin{document}
\maketitle

\begin{abstract}
In this paper we consider a compound Poisson risk model with regularly varying claim sizes. For this model in \cite{AsmussenKlueppelberg:1996} an asymptotic formula for the finite time ruin probability is provided when the time is scaled by the mean excess function. In this paper we derive the rate of convergence for this finite time ruin probability when the claims have a finite 
 second moment.
\end{abstract}
\section{Introduction}\label{S:Intr}
\setcounter{equation}{0}

In this paper, we consider the classical Cram\'er Lundberg risk process with (for convenience)
constant premium inflow $1$,
claims $X_1,X_2,\ldots$ which are iid random variables with distribution $F$ and arrive at
the epochs of  a Poisson process $N_t$ with parameter $\lambda$ and independent of the $X_i$. 
Denote with 
\[
 S_t =\sum_{i=1}^{N_t} X_i-t
\]
the claim surplus process at time $t$ and with
\[
 \tau_u=\inf\{t| u-S_t<0\}
\]
 the time of ruin with starting capital $u$. We are interested in the finite time ruin probability
\[
 \psi(u,t)=\PP(\tau_u<t).
\]
Denote with $\mu=\mean{X}$ and \[
             F_0(x)= \frac 1{\mu} \int_0^x \F(x) \de x 
            \]
the integrated tail distribution of $F$. We assume the usual net profit condition
$\rho=\lambda\mu<1$ ensuring  that the ruin in infinite time does not occur w.p.\ 1.
See for example \cite{AsmussenAlbrecher:2010}.

In \cite{AsmussenKlueppelberg:1996} (see also \cite[Section X.4]{AsmussenAlbrecher:2010}) it is shown that if $\F_0$ is subexponential 
and there exists a non-degenerate random variable $W$ and a function $e(u)$ such that 
\begin{equation}
 \lim_{u\to \infty} \frac{\F_0(u+ xe(u))}{\F_0(u)} =\PP(W>x),\label{mdacond}
\end{equation}
then 
\begin{equation}
 \psi(u,xe(u) )\sim \frac \rho {1-\rho}\F_0(u) \PP\left(\frac{W}{1-\rho} \le x\right)\label{firstorder}
\end{equation}
as $u\to\infty$ (see also \cite{AsmussenHojgaard:1996}, 
\cite{KKM:04} and the discussion in \cite[p.\ 318]{AsmussenAlbrecher:2010} for further work
in this direction). \\
In this paper we want to give asymptotic expressions for the error  in the approximation \eqref{firstorder}.
Condition \eqref{mdacond} (c.f. \cite{EMK:97}) and results on second order asymptotic approximations for compound sums (cf.\ \cite{AlbrecherHippKortschak:09} for a recent survey) imply that we have to expect three different cases: $\F_0$ is regularly varying and has finite mean, $\F_0$ is regularly varying and has infinite mean, $F_0$ is in the maximum domain of attraction of the Gumbel distribution. In this paper we will only consider the first case, where $W$ is regularly varying with finite mean
(see further Assumption \ref{mdacond} below). \\
It should be noted that the our results also have some relevance
for queueing and inventory theory. This is because of the relation between
the Cram\'er-Lundberg model and a dual M/G/1 queue defined by the same
arrival process and service times distributed as the $X_i$: $\psi(u,t)=
\PP(V_t>u)$ where $V_t$ is the workload process in an initially
empty queue (see \cite[ pp.\ 45--48]{AsmussenAlbrecher:2010}). This process is also frequently
used as a storage process model. \\
We start the paper in Section~\ref{S:Sub}  with a survey of recent result on second order
subexponentiasl asymptotics. Section~\ref{S:Prel} then contain the statement of our main result.
In addition we give the outline of the proof, which has many very technical steps (though often the crux is just careful Taylor expansions). This proof in turn is modeled after that of
\cite{AsmussenKlueppelberg:1996}, where the simple and explicit ladder structure of the
Cram\'er-Lundberg process plays a key role. We also give some discussion of the difficulties
in extending to more general models such as L\'evy processes or renewal models.\\
The proofs of the technical estimates omitted in Section~\ref{S:Prel} then occupy the rest
of the paper. A longer version of the paper with some more detail given is available
upon request from the authors.

\section{Subexponential distributions and second order properties}\label{S:Sub}
\setcounter{equation}{0}
In this paper we will assume that the distribution function $F$ of $X$ is regularly varying with index $\alpha$, i.e. 
\[
 \lim_{u\to\infty} \frac{\PP(X>xu)}{\PP(X>u)}= \lim_{u\to\infty} \frac{\F(xu)}{\F(X>u)}=x^{-\alpha}.
\]
For more information about regularly varying we refer to \cite{BGT:87}. Let $X_1,\ldots,X_n$ be iid copies of $X$ denote with $S_n=\sum_{i=1}^n X_i$ and $M_n=\max_{i=1,\ldots,n} X_i$. The regularly varying distributions are a subclass of the subexponential distributions defined through 
\begin{equation}
 \lim_{u\to\infty}\frac{\PP(S_n>u)}{\PP(X_1>u)}=\lim_{u\to\infty}\frac{\PP(M_n>u)}{\PP(X_1>u)}=n.\label{subexponentiality}
\end{equation}
A basic result on second order asymptotics for subexponential distributions concerns the rate of convergence in \eqref{subexponentiality}.\\

If $\mean{X}<\infty$ and $F$ has  a regularly varying density $f$, then it is shown in \cite{OmeyWillekens:1987} that
\begin{equation}
\PP(S_n>u)= n \F(u)+n (n-1) \mean{X_1} f(u) + o\left(f(u)\right).\label{secondorderOW}
\end{equation}
The regularly varying case with $\mean{X}=\infty$ is treated in \cite{OmeyWillekens:1986}.\\
In \cite{BaltrunasOmey:1998} the result \eqref{secondorderOW} is generalized to a wide class of subexponential distributions. Further it is pointed out in \cite{AlbrecherHippKortschak:09}, that a  Taylor expansion shows that \eqref{secondorderOW} is  equivalent to
\[
 \PP(S_n>u)= n \F\left(u-(n-1)\mean{X_1}\right) + o\left(f(u)\right)\,,
\]
which has the natural interpretation that the sum is large if one component is large and the others behave in a normal way.
 One should note that in the cited references  $n$ can be a (light tailed) random variable. Hence by the Pollaczeck-Khinchine formula these results directly translate to second order results for the infinite time ruin probability.\\
Higher order expansions are provided in \cite{BarbeMcCormick:2004} and \cite{BarbeMcCormickZhanga:2006}; for a recent survey of this topic, 
see \cite{AlbrecherHippKortschak:09}.\\ Extensions of these results are given in \cite{DegenLambriggerSegers:2010} where second order properties for the value-at-risk are provided.  \cite{Borovkov:2009} considered the absolute ruin probability in a model where the insurance company can borrow  money. In \cite{K:2011} dependent but tail independent regularly varying random variables are studied, and in \cite{BinChongfengWeidong:2011} second order properties for the value-at-risk,  when the risks are dependent according to an Archimedean copula, are provided.\\
Studies in the subexponential area often use the relation to extreme value theory,
in our case the fact that condition \eqref{mdacond} is equivalent to the condition that $F_0$ is in the maximum domain of attraction of the Fr\'echet extreme value distribution
(see e.g.\ \cite{EMK:97}). However, we will not use this connection.

\section{Preliminaries and main theorem}\label{S:Prel}
\setcounter{equation}{0}

To fix notation  we  present the idea of the proof of \eqref{firstorder} with the notation and the method given in  \cite{AsmussenAlbrecher:2010}. Therefore denote with 
\[
 \tau_+(0)=0,
\quad \tau_+(i)=\inf\{t>\tau_+(i-1)\,:\,S_t>S_{\tau_+(i-1)})\},\quad i\ge 1
\]
the time of the $i$-th ladder step. 
 
 Further denote with $Y_i=S_{\tau_+(i)} -S_{\tau_+(i-1)}$ and $Z_i=S_{\tau_+(i-1)}-S_{\tau_+(i)-} $ the overshoot, resp.\ the capital before each ladder step. 
It is known that the $(Y_i,Z_i)$ form a sequence of iid random vectors with joint distribution given by $\PP(Y>y,Z>z)=\F_0(y+z)$. Denote with 
\[
 K(u)=\inf \{n:\tau_+(n)<\infty,Y_1+\cdots+ Y_n>u\}
\]
the number of ladder steps until the time of ruin and with $\PP^{(u,n)}=\PP(\cdot| \tau(u)<\infty,K(u)=n)$ \\
Denote with $R_t$ a stochastic process  independent of $S_t$ and $R_t\stackrel{d}{=}-S_t$. Let $w(x)=\inf\{t:R_t=x \}$ the first time that the process $R_t$ reaches level $x$.
Under the measure $\PP^{(u,n)}$ the distribution of $\tau (u)$ is the same as the one of $w(Z_1)+\cdots+w(Z_n)$ and $w(Z_1+\cdots+Z_n)$. Hence it follows that for  $Z_1,\ldots,Z_n|K(u)=n$ distributed according to $P^{(u,n)}$, the distribution of $\tau(u)$ is the same as the distribution of 
 $w(Z_1+\cdots+Z_{K(u)})$. So the method of proof for \eqref{firstorder} is first to find the distribution of $Z_1,\ldots,Z_n$ and then find the connection between $w(A)$ and $A$ for some 
random variable $A$.\\
We will use the same ideas to prove our main results. We will work under the following Assumption which will be assumed to hold throughout the paper.
\begin{assum}\label{assumption} Let $X,X_1,X_2,\ldots$ be a sequence of iid random variables with distribution function $F$ having a regularly varying tail with index $\alpha$, a regularly varying density $f$ and Laplace transform $\hat F(s)=\int_0^\infty \e^{-s x}f(x)\de x$. Assume that $\mean{X^2}<\infty$ and that there exists an $M>0$ with $s\hat F(s)<M$ for $\Ree(s)>0$ and $|s|<1$. 
\end{assum}
It follows in particular that, taking $e(u)=u$, the r.v.\ $W$ in \eqref{mdacond} exists and
has tail $\PP(W>y)=(1+y)^{-\alpha+1}$.

\begin{theorem}\label{maintheo}
 Let Assumption \ref{assumption} be fulfilled and define  $e(u)=u$. Then
\begin{align*}
 \psi(u, xe(u)) &=     \frac{\rho\F_0(u+x(1-\rho)e(u))} {(1-\rho)} +\frac {3\mean{X^2}} \mu  \frac{\rho^2\F(u+x(1-\rho)e(u))} {(1-\rho)^2} \\&\quad-\psi(u)\frac{\lambda \mean{X_1^2}}{2e(u)(1-\rho)}\left(\frac{(\alpha-1) } {(1+x(1-\rho))^\alpha}  - \frac {\alpha(\alpha-1)x(1-\rho)} {(1+x (1- \rho))^{\alpha+1}}\right) \\&\quad +o\left( {\F(u)} \right).
\end{align*}
\end{theorem}
\begin{rem}\rm Using $\PP(W>y)=(1+y)^{-\alpha+1}$  and simple calculus, it is easy to see
that the r.h.s.\ of \eqref{firstorder} and the first term in the expansion of $\psi(u, xe(u)) $
in Theorem \ref{maintheo} are both of order $c_1L(u)/u^{\alpha-1}$, with $L(u)$ the slowly
varying function common for $f,F,F_0$ and
$c_1=\rho/\bigl[(1-\rho)\mu(\alpha-1)\bigr]$. The two next terms are, up to constants, both of order
$L(u)/u^{\alpha}$.
\end{rem}
\begin{proof} We give the outline, 
with some lengthy and technical details being given later as Lemmas \ref{lemma:abw}--\ref{lemma:symmetry}
and \ref{assZ}--\ref{lemma:regularityZd}.
From \cite{AsmussenAlbrecher:2010} we get that for $S_n=Y_1+\cdots+Y_n$
\[
 \PP(K(u)=n)=\frac{\rho^n}{\psi(u)} \PP(S_n>u,S_{n-1}\le u).
\]
From Lemma \ref{assZ} we get that
\begin{align*}
&\PP^{(u,n)}(Z_1+\cdots+Z_n>xe(u)) \PP(K(u)=n)\\&= \frac{\rho^n\F_0(u+xe(u))} {\psi(u)} +\frac {3\mean{(n-1)X^2}} \mu  \frac{\rho^n\F(u+xe(u))} {\psi(u)}   +o\left(\frac {\F(u)}{\psi(u)} \right).
\end{align*}
     Summing over $n$ we get that
\begin{align*}
 &\PP^{(u)}(Z_1+\cdots+Z_{K(u)}>xe(u))\\&=    \frac{\rho\F_0(u+xe(u))} {(1-\rho)\psi(u)} +\frac {3\mean{X^2}} \mu  \frac{\rho^2\F(u+xe(u))} {(1-\rho)^2\psi(u)} +o\left(\frac {\F(u)}{\psi(u)} \right).
\end{align*}

From Lemmas \ref{lemma:regularityZ} and \ref{lemma:regularityZd} we get that
\[
W_u=\frac {Z_1+\cdots+Z_n}{x(1-\rho)e(u)}
\] 
fulfills the conditions of Lemma \ref{lemma:abw} and hence
\begin{align*}
 \frac{\psi(u, xe(u))}{\psi(u)} &=     \frac{\rho\F_0(u+x(1-\rho)e(u))} {(1-\rho)\psi(u)} +\frac {3} \mu  \frac{\rho^2\F(u+x(1-\rho)e(u))} {(1-\rho)^2\psi(u)} \\&\quad-\frac{\lambda \mean{X_1^2}}{2e(u)(1-\rho)}\left(\frac{1 }{x(1-\rho)} g_\infty(1)  + \frac 1{x (1- \rho)}g_\infty'(1)\right) \\&\quad +o\left(\frac 1{e(u)}\right) +o\left(\frac {\F(u)}{\psi(u)} \right)
\end{align*}
\end{proof}
\begin{rem}\rm
 The proof of Theorem \ref{maintheo} relies on two observations. First, we used  that the distribution of the sum of the surpluses before each ladder step  has a known distribution, which is related to the distribution that a random sum exceeds a given threshold and hence we can use methods developed for random sums to get second order properties. The second fact that we used is that we know the connection between the time of ruin and the sum of the surpluses. This connection 
 allows to involve the central limit theorem for compound Poisson sums
 (cf.\ Section~\ref{S:Not}) and hence higher order asymptotics can be found. These two properties of compound Poisson processes  are not  straightforward to generalize to more general risk models like renewal models since they heavily rely on the fact that the considered risk process is Markovian. Similarly, 
 the extension to general L\'evy processes meets the difficulty that
 the ladder structure here is more complicated.\\
 Another  interesting extension is to consider the case  where $F$ has finite mean
 but infinite variance. The difficulty here is that the CLT for Poisson sums has to be replaced
 with some sort of stable limit.
\end{rem}

\section{Some notation}\label{S:Not}
\setcounter{equation}{0}

The notation of this section will be used in the rest of the paper without further mentioning.
Recall from Asssumption~\ref{assumption} that $\hat F(s)=\mean{\e^{-sX}}$ is the Laplace transform of the claim size distribution $F$ and let $\kappa(s)=s+\lambda(\hat F(s)-1)$. Then we have
\[
 \mean{\e^{-s S_t}}= \e^{t\kappa(s)}.
\]
We get from  \cite[Lemma XI.3.1]{AsmussenAlbrecher:2010}
\[
 \mean{\e^{-s w(z)}}=\e^{-\kappa^{-1}(s) z}.
\]
For a function $g(x)$ we denote with $\laplace{g}(s)=\int_{0}^\infty \e^{-s x} g(x) \de x$ the Laplace transform. Note that
\[
 \laplace{F}(s)=\frac 1s \hat F(s).
\]
To study the distribution of $w(z)$, note that we can write
\[
 w(z)=z+\sum_{i=1}^{N(z)} E_i\,,
\]
 where  the $E_i$ are iid having the distribution of $E= w(X)$
 (the $E_i$ represent the excursions of $R_t$ away from its running maximum).  Also, as a sample path inspection immediately shows,  $E$ has  the busy period distribution in the usual dual M/G/1 queue (see \cite[pp.\ 45--48]{AsmussenAlbrecher:2010}). 
 Since the Laplace transform is $\hat F_E (s)=\hat F(\kappa^{-1}(s))$, it follows that  
 \begin{align*}
\mean{E}&=\mean{X}/(1-\lambda\mean{X} )=\mean{X}/(1-\rho)\\
\mean{E^2}&=\frac{\mean{X^2}}{(1-\rho)^2}\left(1+\frac {\lambda \mean{X}}{1-\rho}\right)=\frac{\mean{X^2}}{(1-\rho)^3}.                                                                                                                                                                                                                                      \end{align*}
Write $h(z)=w(z)-z(\lambda \mean{E}+1)=w(z)-z/(1-\rho)$ and $U(z)=h(z)/\sqrt{z}$.
By the central limit theorem, $U(z)\to N(0,\lambda\mean{E^2})$. 
Essentially, the claim size density $f$ and the density of $E$ have the same degree
of smoothness. Since we don't want to postulate smoothness conditions on $f$,
we will use smoothing with a normal random variable. Therefore denote with   $N_u$   a normal random variable with mean zero and variance $\sigma^2_u=e(u)^{-4}$.\\

In the proofs of this paper we will often rely on Taylor approximations with remainder 
terms. Therefore we will need to evaluate a function on an interim value which we will denote with $\xi_\Theta$ where $\Theta$ stands for the parameters on which $\xi$ dependents. With  a little abuse of notation we will also use this notation when we use Taylor expansions for a complex function (in this case one would have for the real and the imaginary part a different $\xi$) and when the derivative is not continuous.

\section{The connection between $w(W)$ and  $W$}\label{S:Conn}
\setcounter{equation}{0}

\begin{lemma}\label{lemma:abw}
 Let $W_u$ be a family of random variables with distribution function $G_u(w)$ with $\lim_{u\to\infty}\G_u(w)=\G_\infty(w)=(1+x(1-\rho)w)^{-\alpha+1}$. Further assume that $W_u$ has  a density $g_u(x)$ that is continuously differentiable and $\lim_{u\to\infty}
g_u(w)=g_{\infty}(w)$ as well as $\lim_{u\to\infty}g_u'(w)=g'_{\infty}(w)$. 
Then 
\begin{align*}
&\PP(w( (1-\rho)x e(u) W_u)>xe(u))= \PP(W_u>1)\\&\quad-\frac{\lambda \mean{X_1^2}}{2e(u)(1-\rho)}\left(\frac{1  }{x(1-\rho)} g_\infty(1)  + \frac 1{x (1- \rho)}g_\infty'(1)\right) +o\left(\frac 1{e(u)}\right). 
\end{align*} 
\end{lemma}
\begin{proof}
First consider $W_u>1/(1-\epsilon)$. We get from Lemma \ref{lemma:bound} that there exists an $\delta>0$ with
\begin{align}
&\PP(W_u>1/(1-\epsilon))-\PP(w((1-\rho)xe(u) W_u)>xe(u),W_u>1/(1-\epsilon))\notag\\
&=\int_{1/(1-\epsilon)}^\infty  \PP\left(\frac{h((1-\rho)x e(u) w) }{(1-\rho)x e(u) w} \le \frac{1/w-1}{1-\rho} \right)  \de G_{u}(w)\notag\\
&\le \int_{1/(1-\epsilon)}^\infty\PP\left(\frac{h((1-\rho)x e(u) w) }{(1-\rho)x e(u) w} \le -\frac{\epsilon}{1-\rho} \right)  \de G_{u}(w)\notag\\&
 \le \int_{1/(1-\epsilon)}^\infty  \e^{-\delta e(u)w} \de G_{u}(w)\notag\\
&=o(e(u)^{-1}).\label{enequalitybelow}
\end{align}
For $W_u<1/(1+\epsilon)$ we get by \cite{KM:97}
\begin{align}
&\PP(w((1-\rho)xe(u) W_u)>xe(u),W_u<1/(1+\epsilon))\notag\\
&=\int_0^{1/(1+\epsilon)}  \PP\left(\frac{h((1-\rho)x e(u) w) }{(1-\rho)x e(u) w} > \frac{1/w-1}{1-\rho} \right)  \de G_{u}(w)\notag\\
&
 \sim \int_0^{1/(1+\epsilon)} \lambda (1-\rho)x e(u) w \PP\left(E>(1-w)x e(u) \right)  \de G_{u}(w) \notag\\
&\le \lambda (1-\rho)x e(u)\PP\left(E>\frac{\epsilon}{1+\epsilon}x e(u)\right) \int_0^{1/(1+\epsilon)} w\de G_{u}(w)\notag\\
&=o(e(u)^{-1})\label{enequalityabove}.
\end{align}
Finally we have to consider the case $1/(1+\epsilon)< W_u<1/(1-\epsilon)$.
\begin{align}
&\PP(w((1-\rho)xe(u) W_u)>xe(u),1/(1+\epsilon)<W_u<1/(1-\epsilon))\notag\\
&=\int_{1/(1+\epsilon)}^{1/(1-\epsilon)} \PP\left(\frac{h((1-\rho)x e(u) w) }{(1-\rho)x e(u) w} >\frac{1/w-1}{1-\rho} \right)  \de G_{u}(w)\notag\\
&=\PP(1\le W_u<1/(1-\epsilon))\notag\\
&\quad + \int_{1/(1+\epsilon)}^{1} \PP\left(\frac{h((1-\rho)x e(u) w) }{(1-\rho)x e(u) w} >\frac{1/w-1}{1-\rho} \right)  \de G_{u}(w)\label{int1}\\
&\quad -\int_{1}^{1/(1-\epsilon)} \PP\left(\frac{h((1-\rho)x e(u) w) }{(1-\rho)x e(u) w} \le \frac{1/w-1}{1-\rho} \right)  \de G_{u}(w)\label{int2}.
\end{align}
We start with \eqref{int1}. Denote with 
\begin{equation}\label{definitionx}x(w,u)=\frac{x(1-\rho)}{1+(1-\rho)\frac w{\sqrt{e(u)}}}.\end{equation}
For $N_u$ normally distributed with mean $0$ and variance $\sigma_u^2=(xe(u))^{-4}$ we get from Lemma \ref{lemma:smoothing}
\begin{align}
&\int_{1/(1+\epsilon)}^{1} \PP\left(\frac{h((1-\rho)x e(u) w) }{(1-\rho)x e(u) w} >\frac{1/w-1}{1-\rho} \right)  \de G_{u}(w)\notag\\
&=\frac{1-\rho}{\sqrt{e(u)}} \int_0^{\sqrt{e(u)}\epsilon/(1-\rho)}  \PP\left(\frac{h\left(x(w,u) e(u)\right) }{\sqrt{x(w,u) e(u)}} >w\sqrt{x(w,u)}\right) \frac{g_u\left(\frac 1{1+(1-\rho)\frac w{\sqrt{e(u)}}}\right)}{\left(1+(1-\rho) \frac w{\sqrt{e(u)}}\right)^2}\de w\notag\\
&=\frac{1-\rho}{\sqrt{e(u)}} \int_0^{\sqrt{e(u)}\epsilon/(1-\rho)}  \PP\left(N_u+\frac{h\left(x(w,u) e(u)\right) }{\sqrt{x(w,u) e(u)}} >w\sqrt{x(w,u) }\right)\notag\\
&\quad \times \frac{g_u\left(\frac 1{1+(1-\rho)\frac w{\sqrt{e(u)}}}\right)}{\left(1+(1-\rho) \frac w{\sqrt{e(u)}}\right)^2}\de w+o\left(\frac 1 {e(u)}\right)\notag\\
&=\frac{1-\rho}{\sqrt{e(u)}} \int_0^{\sqrt{e(u)}\epsilon/(1-\rho)}  \PP\left(N_u+\frac{h\left(x(w,u) e(u)\right) }{\sqrt{x(w,u) e(u)}} >w\sqrt{x(w,u) }\right) g_u\left(1\right)\de w\label{equationintegral1}\\
&\quad-\frac{(1-\rho)^2}{e(u)} \int_0^{\sqrt{e(u)}\epsilon/(1-\rho)}  w\PP\left(N_u+\frac{h\left(x(w,u) e(u)\right) }{\sqrt{x(w,u) e(u)}} >w\sqrt{x(w,u) }\right) \label{equationintegral2}\\&
\quad\times\left(\frac{g_u'\left(\frac1{1+\xi_{u,w}}\right)}{(1+\xi_{u,w})^4}+2\frac{g_u\left(\frac1{1+\xi_{u,w}}\right)}{(1+\xi_{u,w})^3}\right)\de w+o\left(\frac 1 {e(u)}\right)\notag.
\end{align}
 We have to evaluate the integrals in \eqref{equationintegral1} and \eqref{equationintegral2} so we split the integrals into $\int_0^M$ and $\int_M^{\sqrt{e(u)}\epsilon/(1-\rho)}$. By Lemma \ref{lemma:asymptoicp1} we get that
\begin{align*}
& \frac{1-\rho}{\sqrt{e(u)}} \int_0^{M}  \PP\left(N_u+\frac{h\left(x(w,u) e(u)\right) }{\sqrt{x(w,u) e(u)}} >w\sqrt{x(w,u) }\right) g_u\left(1\right)\de w\\
 &=\frac{1-\rho}{\sqrt{e(u)}} \int_0^{M}\PP\left(N_u+\frac{h\left(x(0,u) e(u)\right) }{\sqrt{x(0,u) e(u)}} >w\sqrt{x(0,u)}\right)g_u\left(1\right)\de w\\&
 \quad +\frac{(1-\rho)^2}{e(u)} \int_0^{M} w^2 \sqrt{x(1-\rho)} f_{0,\infty}(w\sqrt{x(1-\rho)}) g_u\left(1\right)\de w
\\&\quad+\frac{(1-\rho)^2}{2 e(u)} \lambda \mean{E^2 }  \int_0^{M}   
 wf'_{w,\infty}\left(w\sqrt{x(1-\rho)}\right)g_u\left(1\right)\de w\\&\quad+o(1/e(u)).
\end{align*}
Note that 
\begin{align*}
 &\lim_{M\to\infty} \int_0^{M} w^2 \sqrt{x(1-\rho)} f_{w,\infty}(w\sqrt{x(1-\rho)}) g_\infty\left(1\right)\de w=\frac{g_\infty(1)\lambda\mean{E^2}}{2x(1-\rho)}\\
&\lim_{M\to\infty} \int_0^{M}   
wf'_{w,\infty}\left(w\sqrt{x(1-\rho}\right)g_\infty\left(1\right)\de w=-\frac{g_\infty(1)}{2x(1-\rho)}.\\
&\lim_{M\to \infty}\lim_{u\to \infty} \int_0^{M}  w\PP\left(N_u+\frac{h\left(x(w,u) e(u)\right) }{\sqrt{x(w,u) e(u)}} >w\sqrt{x(w,u) }\right)\\&\quad\times\left(\frac{g_u'(1+\xi_{u,w})}{(1+\xi_{u,w})^2}+2\frac{g_u(1+\xi_{u,w})}{(1+\xi_{u,w})^3}\right)\de w=\frac{\lambda \mean{E^2}}{4x(1-\rho)}\left(g_\infty'(1)+2g_\infty(1)\right).\\
\end{align*} 
For the integral $\int_M^{\sqrt{e(u)}\epsilon/(1-\rho)}$ we get from Lemmas \ref{devnormalnew} and \ref{devfromxzero} that there exist a function $R(M,u) \lesssim \frac{C_M} {e(u)}$ and $C_M\to 0$ as $u\to \infty$ such that the sum of \eqref{equationintegral1} and \eqref{equationintegral2} is the same as
\[
\frac{1-\rho}{\sqrt{e(u)}} \int_M^{\sqrt{e(u)}\epsilon/(1-\rho)}  \PP\left(N_u+\frac{h\left(x(0,u) e(u)\right) }{\sqrt{x(0,u) e(u)}} >w\sqrt{x(0,u) }\right) g_\infty\left(1\right)\de w + R(M,u).
\]
With Lemmas \ref{lemma:asymptoicm1}, \ref{devnormalnew1} and \ref{devfromxzerom}, we can get analogously the asymptotic of \eqref{int2},
so that we are left with the integrals 
\begin{align*}
 &\int_0^{\sqrt{e(u)}\epsilon/(1-\rho)}  \PP\left(N_u+\frac{h\left(x(0,u) e(u)\right) }{\sqrt{x(0,u) e(u)}} >w\sqrt{x(0,u) }\right)\de w\\
&- \int_0^{\sqrt{e(u)}\epsilon/(1-\rho)}  \PP\left(N_u+\frac{h\left(x(0,u) e(u)\right) }{\sqrt{x(0,u) e(u)}} \le-w\sqrt{x(0,u) }\right)\de w.
\end{align*}
From Lemma \ref{lemma:symmetry} we get that the last equation  is asymptotically negligibility and hence the Lemma follows.

\end{proof}

\begin{lemma}\label{lemma:asymptoicp1}Under Assumption \ref{assumption} we get uniformly for $w<M$ that
\begin{align*}
 \PP\left(N_u+\frac{h\left(x(w,u) e(u)\right) }{\sqrt{x(w,u) e(u)}} >w\sqrt{x(w,u)}\right)
 &=\PP\left(N_u+\frac{h\left(x(0,u) e(u)\right) }{\sqrt{x(0,u) e(u)}} >w\sqrt{x(0,u)}\right)\\&
 \quad +w^2 \sqrt{x(1-\rho)}\frac {1-\rho}{\sqrt{e(u)}} f_{0,\infty}(w\sqrt{x(1-\rho)}) 
\\&\quad+ \frac{(1-\rho)w}{2\sqrt{e(u)}} 
 \lambda \mean{E^2 }f'_{w,\infty}\left(w\sqrt{x(1-\rho}\right)\\&\quad+o(1/\sqrt{e(u)}).
\end{align*}
 where $x(w,u)$ is defined by \eqref{definitionx} and $f_{w,\infty}$ is the density of a normal distribution with mean $0$ and variance $\lambda \mean{E^2}$.
\end{lemma}
\begin{proof}  Denote with 
$f_{w,u}(x)$ the density of 
\[
 Z_u =N_u+\frac{h\left(x(w,u) e(u)\right) }{\sqrt{x(w,u) e(u)}}.
\]
and with $\hat f_{w,u}(x)$ the density of 
\[
\hat Z_u =\hat N_u+\frac{h\left((x(0,u)-x(w,u)) e(u)\right) }{\sqrt{(x(0,u)-x(w,u)) e(u)}}.
\]
where $\hat N_u$ is an independent copy of $N(u)$. $Z_u$ and $\hat Z_u$ are independent. 
Since $x(w,u)$ is monotonically decreasing in $w$, we get that
\begin{align*}
 &\PP\left(N_u+\frac{h\left(x(0,u) e(u)\right) }{\sqrt{x(0,u) e(u)}} >w\frac{x(w,u)}{\sqrt{x(0,u)}}\right)\\
&= \PP\left(N_u+\frac{h\left(x(w,u) e(u)\right) +h\left((x(0,u)-x(w,u)) e(u)\right)}{\sqrt{x(0,u) e(u)}} >w\frac{x(w,u)}{\sqrt{x(0,u)}}\right)\\
&= \PP\left(Z_u +\sqrt{\frac{(1-\rho)w}{\sqrt{e(u)}}} \hat Z_u>w\sqrt{x(w,u)}\right)\\
&=\PP\left(Z_u >w\sqrt{x(w,u)} \right)+ \sqrt{\frac{(1-\rho)w}{\sqrt{e(u)}}} \mean{\hat Z_u} f_{w,u}(w\sqrt{x(w,u)})\\
&\quad - \frac{(1-\rho)w}{2\sqrt{e(u)}}  \mean{\hat Z_u^2 f'_{w,u}(w\sqrt{x(w,u)}+\xi_{w,u})}\\
&=\PP\left(Z_u >w\sqrt{x(w,u)} \right)- \frac{(1-\rho)w}{2\sqrt{e(u)}} 
 \mean{\hat Z_u^2 }f'_{w,\infty}\left(w\sqrt{x(1-\rho}\right)+o\left(\frac 1{\sqrt{e(u)}}\right),
\end{align*}
here the last equality follows by bounded convergence.
Finally note that
\begin{align*}
&\PP\left(N_u+\frac{h\left(x(0,u) e(u)\right) }{\sqrt{x(0,u) e(u)}} >w\frac{x(w,u)}{\sqrt{x(0,u)}}\right)\\
 &=\PP\left(N_u+\frac{h\left(x(0,u) e(u)\right) }{\sqrt{x(0,u) e(u)}} >w\sqrt{x(0,u)}\right)\\&\quad 
 +\left(w\frac{x(0,u)-x(w,u)}{\sqrt{x(0,u)}}\right)f_{0,u}(w\sqrt{x(0,u)}+\xi_{u,w})\\
 &=\PP\left(N_u+\frac{h\left(x(0,u) e(u)\right) }{\sqrt{x(0,u) e(u)}} >w\sqrt{x(0,u)}\right)\\&\quad 
 +w^2 \sqrt{x(1-\rho)}\frac {1-\rho}{\sqrt{e(u)}} f_{0,\infty}(w\sqrt{x(0,u)}) +o\left(\frac 1{\sqrt{e(u)}}\right).
\end{align*}
 \end{proof}


\begin{lemma}\label{devnormalnew}Under Assumption \ref{assumption} we get that for every $c>0$ uniformly for $M\le w<c\sqrt{e(u)}$ that

\begin{align*}
 &\left|\PP\left(N_u+\frac{h\left(x(w,u) e(u)\right) }{\sqrt{x(w,u) e(u)}} >w\sqrt{x(w,u)}\right)\right.\left.-\PP\left(N_u+\frac{h\left(x(0,u) e(u)\right) }{\sqrt{x(0,u) e(u)}} >w\frac{x(w,u)}{\sqrt{x(0,u)}}\right)\right|\\&\le \frac {C_1} { w^2\sqrt{ e(u)}} +  C_2w  \sqrt{e(u)}  \PP\left(E>\frac{\epsilon_1}2w \sqrt{e(u)}\sqrt{(1-\rho) x(w,u)}  \right)+o\left(\frac 1{e(u)}\right) 
\end{align*}
where $x(w,u)$ is defined by \eqref{definitionx}.
\end{lemma}
\begin{proof}We use the notation of the proof of Lemma \ref{lemma:asymptoicp1}. \\ 
Choose an $0<\epsilon_1<1$. Since $x(w,u)$ is monotonically decreasing in $w$, we get that
\begin{align*}
 &\PP\left(N_u+\frac{h\left(x(0,u) e(u)\right) }{\sqrt{x(0,u) e(u)}} >w\frac{x(w,u)}{\sqrt{x(0,u)}}\right)\\
&= \PP\left(Z_u +\sqrt{\frac{(1-\rho)w}{\sqrt{e(u)}}} \hat Z_u>w\sqrt{x(w,u)},|\hat Z_u|\le \epsilon_1 \sqrt {w\sqrt{e(u)}} \right)\\
&\quad +\PP\left(Z_u +\sqrt{\frac{(1-\rho)w}{\sqrt{e(u)}}} \hat Z_u>w\sqrt{x(w,u)},|\hat Z_u|> \epsilon_1 \sqrt {w\sqrt{e(u)}} \right).
\end{align*}
Note that
\begin{align*}
 &\PP\left(Z_u +\sqrt{\frac{(1-\rho)w}{\sqrt{e(u)}}} \hat Z_u>w\sqrt{x(w,u)},|\hat Z_u|\le \epsilon_1 \sqrt {w\sqrt{e(u)}} \right)\\
&=\PP\left(Z_u >w\sqrt{x(w,u)},|\hat Z_u|\le \epsilon_1 \sqrt {w\sqrt{e(u)}} \right)\\
&+ \sqrt{\frac{(1-\rho)w}{\sqrt{e(u)}}} \mean{\hat Z_u\indicator{ |\hat Z_u|\le \epsilon_1 \sqrt {w\sqrt{e(u)}}}} f_{w,u}(w\sqrt{x(w,u)})\\
&\quad + \frac{(1-\rho)w}{\sqrt{e(u)}}  \mean{\hat Z_u^2 f'_{w,u}(w\sqrt{x(w,u)}+\xi_{w,u})\indicator{ |\hat Z_u|\le \epsilon_1 \sqrt {w\sqrt{e(u)}}}}.
\end{align*}
Since $\mean{\hat Z_u}=0$
we get that
\begin{align*}
\left|\mean{\hat Z_u\indicator{ |\hat Z_u|\le \epsilon_1 \sqrt {w\sqrt{e(u)}}}}\right|& =\left|\mean{\hat Z_u\indicator{ |\hat Z_u|> \epsilon_1 \sqrt {w\sqrt{e(u)}}}}\right|\\
&\le \frac{1}{ \epsilon_1\sqrt {w\sqrt{e(u)}}}\left|\mean{\hat Z_u^2\indicator{ |\hat Z_u|> \epsilon_1 \sqrt {w\sqrt{e(u)}}}}\right|.
\end{align*}
By Lemma \ref{lemma:boundderivative} $x^2 f_{w,u}(x)$ is  bounded and hence for some $c_1>0$
\[
 \sqrt{\frac{(1-\rho)w}{\sqrt{e(u)}}} \mean{\hat Z_u\indicator{ |\hat Z_u|\le \epsilon_1 \sqrt {w\sqrt{e(u)}}}} f_{w,u}(w\sqrt{x(w,u)}))\le c_1 \mean{\hat Z_u^2} \frac1{w^2 \sqrt{e(u)}}.
\]
Denote with
\begin{align*}
 a=\left(\frac{(1-\rho) x}{1+c (1-\rho)}-\epsilon_1\sqrt{1-\rho}\right)
\quad \text{and}\quad b=\left({(1-\rho) x}+\epsilon_1\sqrt{1-\rho}\right).
\end{align*}
We will assume that $\epsilon_1$ is chosen such that $a>0$.
From
\begin{align*}
\mean{\hat Z_u^2 f'_{w,u}(w\sqrt{x(w,u)})+\xi_{w,u})\indicator{| \hat Z_u|\le \epsilon_1 \sqrt {w\sqrt{e(u)}}}}
\le \mean{\hat Z_u^2 } \sup_{aw<x<b w } f'_{w,u}(x)
\end{align*}
and Lemma \ref{lemma:boundderivative} we get that  $\sup_{aw<x<b w } f'_{w,u}(x)\le c_2/ w^{3}$ and for some $c_3>0$
\[
 \frac{(1-\rho)w}{\sqrt{e(u)}}  \mean{\hat Z_u^2 f'_{w,u}(w\sqrt{x(w,u)})+\xi_{w,u})\indicator{ |\hat Z_u|\le \epsilon_1 \sqrt {w\sqrt{e(u)}}}}\le c_3 \mean{\hat Z_u^2} \frac1{w^2 \sqrt{e(u)}}.
\]

Further we have
with Lemma \ref{lemma:bound} and $\PP(|X+Y|>u)\le\PP(|X|>u/2)+\PP(|Y|>u/2)$ that for a standard normal distributed random variable $\mathcal N$
\begin{align*}
 & \PP\left(Z_u +\sqrt{\frac{(1-\rho)w}{\sqrt{e(u)}}} \hat Z_u>w\sqrt{x(w,u)},|\hat Z_u|> \epsilon_1 \sqrt {w\sqrt{e(u)}} \right)
 \\&\le\PP\left(|\hat Z_u|> \epsilon_1 \sqrt {w\sqrt{e(u)}} \right)\\&\le  C_2w(1-\rho) x(w,u) \sqrt{e(u)} \sqrt{x(w,u)} \PP\left(E>\frac{\epsilon_1}2w \sqrt{e(u)}\sqrt{(1-\rho) x(w,u)}  \right)\\&\quad + e^{-\delta\frac{\epsilon_1}4w \sqrt{e(u)}\sqrt{(1-\rho) x(w,u)} }+\PP\left(|\mathcal N|>\frac{\epsilon_1}2 \sqrt {w\sqrt{e(u)}}\right).
 \end{align*}

 \end{proof}
\begin{lemma}\label{devfromxzero}
Under Assumption \ref{assumption} we get that
\begin{align*}
 &\int_M^{c\sqrt{e(u)}}\PP\left(N_u+\frac{h\left(x(0,u) e(u)\right) }{\sqrt{x(0,u) e(u)}} >w\frac{x(w,u)}{\sqrt{x(0,u)}}\right)\de w\\&=\int_M^{c\sqrt{e(u)}}\PP\left(N_u+\frac{h\left(x(0,u) e(u)\right) }{\sqrt{x(0,u) e(u)}} >w \sqrt{x(0,u)}\right)\de w+R(u,M),
\end{align*}
where \[R(u,M)\lesssim 
       \frac{C_M}{\sqrt{e(u)}}
      \]
and $C_M  \to 0$ as $M\to \infty$.
\end{lemma}
\begin{proof} By substitution we get that
\begin{align}
 & \int_M^{c\sqrt{e(u)}}  \PP\left(N_u+\frac{h\left(x(0,u) e(u)\right) }{\sqrt{x(0,u) e(u)}} >w\frac{x(w,u)}{\sqrt{x(0,u)}}\right) \de w\notag\\
&\int_{\frac{M}{1+M\frac{1-\rho}{\sqrt{e(u)}}}}^{\frac{c\sqrt{e(u)}}{1+c(1-\rho)}} \frac{1}{1-w\frac{1-\rho}{\sqrt{e(u)} }} \PP\left(N_u+\frac{h\left(x(0,u) e(u)\right) }
{\sqrt{x(0,u) e(u)}} >w\sqrt{x(0,u)}\right) \de w\label{firstint}\\
&\quad+\frac {1-\rho}{\sqrt{e(u)}} \int_{\frac{M}{1+M\frac{1-\rho}{\sqrt{e(u)}}}}^{\frac{c\sqrt{e(u)}}{1+c(1-\rho)}}\frac{w}{\left(1-w\frac{1-\rho}{\sqrt{e(u)} }\right)^2} \PP\left(N_u+\frac{h\left(x(0,u) e(u)\right) }{\sqrt{x(0,u) e(u)}} >w\sqrt{x(0,u)}\right) \de w\label{secondint}
\end{align}
 \eqref{secondint} can be bounded by
\[
 \frac {1-\rho}{\sqrt{e(u)}} (1+c(1-\rho))^2 \int_{\frac{M}{1+M\frac{1-\rho}{\sqrt{e(u)}}}} ^{\infty} w \PP\left(N_u+\frac{h\left(x(0,u) e(u)\right) }{\sqrt{x(0,u) e(u)}} >w\sqrt{x(0,u)}\right) \de w\sim \frac {c_M}{\sqrt{e(u)}}. 
\]
where $c_M\to0$ as $M\to \infty$. For \eqref{firstint} we have
\begin{align}
& \int_{\frac{M}{1+M\frac{1-\rho}{\sqrt{e(u)}}}}^{\frac{c\sqrt{e(u)}}{1+c(1-\rho)}} \frac{1}{1-w\frac{1-\rho}{\sqrt{e(u)} }} \PP\left(N_u+\frac{h\left(x(0,u) e(u)\right) }
{\sqrt{x(0,u) e(u)}} >w\sqrt{x(0,u)}\right) \de w \notag\\
&= \int_{\frac{M}{1+M\frac{1-\rho}{\sqrt{e(u)}}}}^{\frac{c\sqrt{e(u)}}{1+c(1-\rho)}} \frac{1-\rho}{\sqrt{e(u)} } \frac{w}{1-w\frac{1-\rho}{\sqrt{e(u)} }} \PP\left(N_u+\frac{h\left(x(0,u) e(u)\right) }
{\sqrt{x(0,u) e(u)}} >w\sqrt{x(0,u)}\right) \de w\label{firstint1}\\
&\quad+\int_{\frac{M}{1+M\frac{1-\rho}{\sqrt{e(u)}}}}^{\frac{c\sqrt{e(u)}}{1+c(1-\rho)}}  \PP\left(N_u+\frac{h\left(x(0,u) e(u)\right) }
{\sqrt{x(0,u) e(u)}} >w\sqrt{x(0,u)}\right) \de w\label{secondint1}.
\end{align}
Here \eqref{firstint1} can be bounded similar to \eqref{secondint}. The integral \eqref{secondint1} split into
\begin{align*}
&
\int_{M}^{{c\sqrt{e(u)}}} +\int_{\frac{M}{1+M\frac{1-\rho}{\sqrt{e(u)}}}}^{M} -\int_{\frac{c\sqrt{e(u)}}{1+c(1-\rho)}}^{{c\sqrt{e(u)}}} \PP\left(N_u+\frac{h\left(x(0,u) e(u)\right) }
{\sqrt{x(0,u) e(u)}} >w\sqrt{x(0,u)}\right) \de w.
\end{align*}
where the last integral can be bounded as in \eqref{enequalityabove}. Further
\begin{align*}
& \int_{\frac{M}{1+M\frac{1-\rho}{\sqrt{e(u)}}}}^{M}\PP\left(N_u+\frac{h\left(x(0,u) e(u)\right) }
{\sqrt{x(0,u) e(u)}} >w\sqrt{x(0,u)}\right) \de w \\&\le
\frac{1} {\sqrt{e(u)}}\frac{M^2 (1-\rho)}{1+M\frac{1-\rho}{\sqrt{e(u)}}} \PP\left(N_u+\frac{h\left(x(0,u) e(u)\right) }
{\sqrt{x(0,u) e(u)}} > \frac{M\sqrt{x(0,u)}}{1+M\frac{1-\rho}{\sqrt{e(u)}}}\right)\\
&\sim \frac{M^2 (1-\rho)} {\sqrt{e(u)}} \left(1-\Phi\left(\frac{ M\sqrt{x(1-\rho)}}{\sqrt{\lambda\mean{E^2}}}\right)\right).
\end{align*}
Hence the Lemma follows.

\end{proof}

We now provide the similar Lemmas for \eqref{int2}. We will skip the proofs, since apart from some obvious modifications  they are similar to the proofs of Lemmas \ref{lemma:asymptoicm1}, \ref{devnormalnew1} and \ref{devfromxzerom}.

\begin{lemma}\label{lemma:asymptoicm1}Under Assumption \ref{assumption} we get that for every $c>0$ uniformly for $w<M$ that
\begin{align*}
& \PP\left(N_u+\frac{h\left(x(-w,u) e(u)\right) }{\sqrt{x(w,u) e(u)}} \le -w\sqrt{x(-w,u)}\right)\\
 &=\PP\left(N_u+\frac{h\left(x(0,u) e(u)\right) }{\sqrt{x(0,u) e(u)}} \le -w\sqrt{x(0,u)}\right)\\&  - w^2\sqrt{x(1-\rho)} \frac{ (1-\rho)}{\sqrt{e(u)}}f_{0,\infty}(-w\sqrt{x(1-\rho)})
\\&+ \frac{ w(1-\rho)}{2\sqrt{e(u)}}
 \mean{\hat Z_u^2 }f'_{0,\infty}\left( -w \sqrt{x(1-\rho)}\right)+o\left(\frac {1}{\sqrt{e(u)}}\right).
\end{align*}
 where $x(w,u)$ is defined by \eqref{definitionx} and $f_{\infty}$ is the density of a normal distribution with mean $0$ and variance $\lambda \mean{E^2}$.
\end{lemma}
\begin{lemma}\label{devnormalnew1}Under Assumption \ref{assumption} we get that for every $c>0$ uniformly for $M\le w<c\sqrt{e(u)}$ that
\begin{align*}
 &\left|\PP\left(N_u+\frac{h\left(x(-w,u) e(u)\right) }{\sqrt{x(-w,u) e(u)}} \le -w\sqrt{x(-w,u)}\right)\right.\left.-\PP\left(N_u+\frac{h\left(x(0,u) e(u)\right) }{\sqrt{x(0,u) e(u)}} \le -w\frac{x(-w,u)}{\sqrt{x(0,u)}}\right)\right|\\&\le \frac {C_1} { w^2\sqrt{ e(u)}} +  C_2w \sqrt{e(u)} \PP\left(E>\frac{\epsilon_1}2w \sqrt{e(u)}\sqrt{(1-\rho) x(w,u)}  \right)+o\left(\frac 1{e(u)}\right)\\
\end{align*}
where $x(w,u)$ is defined by \eqref{definitionx}.
\end{lemma}

\begin{lemma}\label{devfromxzerom}Under Assumption \ref{assumption} we get that
\begin{align*}
 &\int_M^{c\sqrt{e(u)}}\PP\left(N_u+\frac{h\left(x(0,u) e(u)\right) }{\sqrt{x(0,u) e(u)}} \le-w\frac{x(-w,u)}{\sqrt{x(0,u)}}\right)\de w\\&=\int_M^{c\sqrt{e(u)}}\PP\left(N_u+\frac{h\left(x(0,u) e(u)\right) }{\sqrt{x(0,u) e(u)}} \le -w \sqrt{x(0,u)}\right)\de w+R(u,M),
\end{align*}
where \[R(u,M)\lesssim 
       \frac{C_M}{\sqrt{e(u)}}
      \]
and $C_M  \to 0$ as $M\to \infty$.
\end{lemma}

\begin{lemma} \label{lemma:symmetry}
 Under Assumption \ref{assumption} we get that
\begin{align*}
\int_{0}^{c\sqrt{xe(u)}} \PP\left(N_u +\frac{h(xe(u))}{\sqrt{x e(u)}} \le -w\right)- \PP\left(N_u +\frac{h(xe(u))}{\sqrt{x e(u)}} > w\right) \de w =o\left(\frac 1{\sqrt{e(u)}}\right).
\end{align*}
\end{lemma}

\begin{proof}
Denote with $\hat \chi_u $ the characteristic function of $N_u +{h(xe(u))}/{\sqrt{x e(u)}}$.
From the  Gil-Pelaez inversion formula we get that (c.f. \cite{GilPelaez:1951}, \cite{Wendel:1961})
\begin{align*}
&\int_0^{c\sqrt{xe(u)}} \PP\left(N_u +\frac{h(xe(u))}{\sqrt{x e(u)}} > w\right)-\PP\left(N_u +\frac{h(xe(u))}{\sqrt{x e(u)}} \le w\right)\de w\\
&=\int_0^{c\sqrt{xe(u)}}\frac 1 \pi \int_{0}^\infty\frac 1 s \Ime\left( \e^{-\I w s}\hat \chi_u(s)\right)  \de s+\frac 1 \pi \int_{0}^\infty\frac 1 s \Ime\left( \e^{\I w s}\hat \chi_u(s)\right)  \de s\de w\\
&=\int_0^{c\sqrt{xe(u)}} \frac 2 \pi  \int_{0}^\infty \frac{\cos(ws )}{s} \Ime (\hat \chi_u(s)) \de s\de w\\
&=\frac 2 \pi  \int_{0}^ {\epsilon\sqrt{xe(u)}} \frac{\sin(c\sqrt{x(e(u))} s )}{s^2} \Ime (\hat \chi_u(s)) \de s
+\frac 2 \pi  \int_{\epsilon\sqrt{xe(u)}}^\infty \frac{\sin(c\sqrt{x(e(u))} s )}{s^2} \Ime (\hat \chi_u(s)) \de s\\
&=I_1(u)+I_2(u)\,,
\end{align*}
where $\epsilon$ is chosen 
such that for $|s|<\epsilon$,  $-\Ree(\chi_E''(s))\ge \delta_1$ for some $\delta_1>0$. Since there exists a $\delta>0$ such that for all $|s|>\epsilon$, $\Ree(1-\chi_E(s))\ge \delta$ ($E_i$ is non lattice). We get 
for $s>\epsilon\sqrt{xe(u)}$
\[
 |\Ime (\hat \chi_u(s))|\le \e^{-\frac {s^2 \sigma_u^2}2} \e^{-\delta\lambda x e(u)}. 
\]
and hence $I_2(u)$ goes to $0$ faster than any power of $e(u)$. Denote with
\begin{align*}
 A_1(s,u)&=\lambda xe(u)\int_0^\infty \cos\left(\frac{s t}{\sqrt{xe(u)}} \right)-1\ \de F_E(t),\\
 A_2(s,u)&= \lambda\sqrt{xe(u)}\int_0^\infty \sqrt{xe(u)} \sin\left(\frac{s t}{\sqrt{xe(u)}} \right)- s t\ \de F_E(t) 
\end{align*}

To get a bound for $I_1(u)$ we get from Lemma \ref{lemma:sinuintegral} that we have to study the derivative of 
\begin{align*}
\frac 1{s^2} \Ime(\chi_u(s)) =\frac 1{s^2} \e^{-\frac {s^2\sigma_u^2} 2} 
\sin\left(A_2(s,u))\right) \e^{A_1(s,u)}
\end{align*}
which is the sum of $D_1$, $D_2$ and $D_3$ given by
\begin{align*}
D_1&=\frac {\sigma_u^2}{s} \e^{-\frac {s^2\sigma_u^2} 2} 
\sin\left(A_2(s,u))\right) \e^{A_1(s,u)}\displaybreak[0]\\
D_2&=\frac 1{s^2} \e^{-\frac {s^2\sigma_u^2} 2} 
\sin\left(A_2(s,u))\right) \e^{A_1(s,u)}\left( 
 \lambda xe(u)\int_0^\infty \frac{-t}{\sqrt{xe(u)}}\sin \left(\frac{s t}{\sqrt{xe(u)}} \right)\ \de F_E(t)\right)\displaybreak[0]\\
D_3&=\frac 1{s^2} \e^{-\frac {s^2\sigma_u^2} 2} 
\e^{-\frac {s^2\sigma_u^2} 2} \Bigg\{\cos\left( 
  A_2(s,u) 
\right)\left( 
  \lambda\sqrt{xe(u)}\int_0^\infty t\left(\cos\left(\frac{s t}{\sqrt{xe(u)}} \right)- 1 \right)\ \de F_E(t) 
\right)\\
&\quad\quad\quad\quad\quad\quad\quad\quad\quad\quad\quad\quad-\frac 2s \sin\left( 
  A_2(s,u)\right) \Bigg\}
\end{align*}
Note that
\begin{align*}
  - \frac{A_1(s,u)}{\lambda  xe(u)}=\Ree\left(\chi_E\left(\frac{s }{\sqrt{xe(u)}}\right)\right)-1=-\frac{s^2}{2xe(u)}\Ree\left(\chi''_E(\xi_{s,u})\right)\ge \delta_1 \frac{s^2}{2xe(u)}.
\end{align*}
Further note that 
\begin{align} 
A_2(s,u)=- s^2 \int_0^\infty t^2 \sin\left(\xi_{s,u,t} \right)\de F_E(t),\label{abschaetzungmomemt}
\end{align}
where $0<\xi_{s,u,t}<\frac{s t}{\sqrt{xe(u)}}$. Now for $s\le(xe(u))^{1/4}$
we have that
\begin{align}
\left|\int_0^\infty t^2 \sin\left(\xi_{s,u,t} \right)\de t\right|&=\left|\int_0^{(xe(u))^{1/8}} t^2 \sin\left(\xi_{s,u,t} \right)\de F_E(t)\right|+\left|\int_{(xe(u))^{1/8}}^\infty t^2 \sin\left(\xi_{s,u,t} \right)\de F_E(t)\right|\notag\\
 &\le \int_0^{(xe(u))^{1/8}} t^2 \sin\left((xe(u))^{-1/8} \right)\de F_E(t)+\int_{(xe(u))^{1/8}}^\infty t^2 \de F_E(t)\notag\\
&\le \sin\left((xe(u))^{-1/8} \right)\mean{E^2}+\int_{(xe(u))^{1/8}}^\infty t^2 \de F_E(t)\to 0\label{abschaetzungmomemt1}
\end{align}
as $u\to\infty$.
Hence for every $\epsilon_1>0$ there exists an $u_0$ such that for $u>u_0$ (note that $|\sin(t)|\le t$).
\begin{align*}
  |D_1|&\le \begin{cases}
 \mean{E^2} s\exp(-\frac{\lambda \delta_1 s^2 }{2})  & s>(xe(u))^{1/4}\\
\epsilon_1 s\exp(-\frac{\lambda \delta_1 s^2 }{2})  & s\le(xe(u))^{1/4}
 \end{cases}\\ |D_2|&\le \begin{cases}
 s \lambda \mean{E^2}^2 \exp(-\frac{\lambda \delta_1 s^2 }{2})  & s>(xe(u))^{1/4}\\
\epsilon_1 s \lambda \mean{E^2}\exp(-\frac{\lambda \delta_1 s^2 }{2})  & s\le(xe(u))^{1/4}
\end{cases}.
\end{align*}
It follows that
\[
 \int_0^{\epsilon \sqrt{x e(u)}} |D_1|+|D_2| \de s = o\left(1\right). 
\]
At last we have to bound $\int_0^1 |D_3|\de s+ \int_1^{\epsilon \sqrt{x e(u)}} |D_3|\de s$.
Since
\[
 \lambda\sqrt{xe(u)}\int_0^\infty t\left(\cos\left(\frac{s t}{\sqrt{xe(u)}} \right)- 1 \right)\ \de F_E(t)=-\lambda s \int_0^\infty t^2\sin(\xi_{s,u,x})\de s.
\]
We get with the same method as above
\[
 \int_1^{\epsilon \sqrt{x e(u)}} |D_3|\de s=o\left(1\right).
\]
For $0<s<1$ we get with \eqref{abschaetzungmomemt} and \eqref{abschaetzungmomemt1} that for large enough $u$ 
\begin{align*}
|D_3|&\le\frac 1{s^2}\Bigg|\cos\left( A_2(s,u) 
\right)\left( 
  \lambda\sqrt{xe(u)}\int_0^\infty t\left(\cos\left(\frac{s t}{\sqrt{xe(u)}} \right)- 1 \right)\ \de F_E(t) 
\right)-\frac 2s \sin\left( 
A_2(s,u) 
\right) \Bigg|\\
&=\frac {\lambda\sqrt{xe(u)}} {s^2}\Bigg|\int_0^\infty t\left(\cos\left(\frac{s t}{\sqrt{xe(u)}} \right)- 1 \right)\ \de F_E(t)\\
&\quad\quad- \frac 2s  
  \int_0^\infty \sqrt{xe(u)} \sin\left(\frac{s t}{\sqrt{xe(u)}} \right)- s t\ \de F_E(t) 
\Bigg|+ o(1).\\
&\le 4 \frac {\lambda\sqrt{xe(u)}} {s^2}\int_{\frac 1s \sqrt{x(e(u))} }^\infty t \de F_E(t)+ \frac{2\sin(1)-3\cos(1)}{2\sqrt{x e(u)}} \int_{0}^{\frac 1s \sqrt{x(e(u))} } t^3  \de F_E(t) +o(1).
\end{align*}
It is left to show that $\int_0^1 \de s$ of the last equation is $o(1)$.
From Karamata's Theorem it follows that 
\begin{align*}
\int_0^1  4 \frac {\lambda\sqrt{xe(u)}} {s^2}\int_{\frac 1s \sqrt{x(e(u))} }^\infty t \de F_E(t)\de s
&\sim  c \int_0^1 \frac {\lambda{xe(u)}} {s^3} \F_E\left(\frac 1s \sqrt{x(e(u))}\right)\de s\\
&=  c \int_{\sqrt{xe(u)}}^\infty s \F_E\left(s\right)\de s=o(1).\\
\end{align*}
If $\mean{E^3}<\infty$ then the Lemma follows. If $\mean{E^3}=\infty$ and $\alpha \not =3$ then
\begin{align*}
 \int_0^1 \frac{1}{\sqrt{x e(u)}} \int_{0}^{\frac 1s \sqrt{x(e(u))} } t^3  \de F_E(t)\de s \sim c \int_0^1 \frac {\lambda{xe(u)}} {s^3} \F_E\left(\frac 1s \sqrt{x(e(u))}\right)\de s=o(1).
\end{align*}
If $\alpha=3$ and  $\mean{E^3}=\infty$ then the integral is  asymptotically less as when we replace $t^3$ with $t^{3.5}$ and the Lemma follows with the same argument.
\end{proof}

\section{The asymptotics of the $Z$}\label{sec:assZ}
\setcounter{equation}{0}

In what follows we will denote with $S_{n}=\sum_{i=1}^n Y_i$, $M_n=\max_{1\le i\le n} Y_i$,  $\hat S_{n}=\sum_{i=1}^n Z_i$, and with $\hat M_n=\max_{1\le i\le n} Z_i$   
\begin{lemma}\label{assZ}
 Let $(Y_1,Z_1),\ldots(Y_1,Z_1)$ be iid vectors with distribution $F_0(y+z)$ where $F_0=\frac 1 \mu \int_0^x\F(t)\de t$ and $F$ fulfill Assumption \ref{assumption}.
Then
\begin{align*}
 &\PP(S_n>u,S_{n-1}\le u,\hat S_n>xe(u))\sim \F_0(u+xe(u))\\&\quad+\frac 1 \mu \mean{S_{n-1}+\hat S_{n-1}+(n-1)Y_n} \F(u+xe(u))\\
&=F_0(u+xe(u))+\frac {3(n-1)\mean{X^2}} \mu  \F(u+xe(u)).
\end{align*}
Further for all $\epsilon>0$ there exists a constant $M$ such that for all $n$  
\[
 \PP(S_n>u,S_{n-1}\le u,Z>x(e(u)))-\F_0(u+xe(u)) \le  M(1+\epsilon)^n\F(u).
\]
\end{lemma}

\begin{proof}
 Note that
\[
 \sum_{i=1}^n \PP(S_n>u,S_{n-1}\le u,\hat S_n>xe(u), M_n=Y_i).
\]
At first we consider $\{M_n=X_n\}$.
We have that
\begin{align*}
 &\PP(S_n>u,S_{n-1}\le u,\hat S_n>xe(u), M_n=Y_n)\\&=\PP(S_n>u,S_{n-1}\le u,\hat S_n>xe(u), M_n=Y_n,S_{n-1}>u/2)\\
&\quad+\PP(S_n>u,\hat S_n>xe(u), M_n=Y_n,S_{n-1}\le u/2,\hat S_{n-1}>u/2)\\
&\quad+\PP(S_n>u,\hat S_n>xe(u), M_n=Y_n,S_{n-1}\le u/2,\hat S_{n-1}\le u/2).
\end{align*}
Since $F$ is  regularly varying case we get
\begin{align*}
& \PP(S_n>u,S_{n-1}\le u,\hat S_n>xe(u), M_n=Y_n,S_{n-1}>u/2)\\
&\le \PP(Y_n>u/n)\PP(S_{n-1}>u/2)\\
&\le K (2n)^{\alpha+d\epsilon} (1+\epsilon)^n \F_0(u)^2
\end{align*}
and 
\begin{align*}
& \PP(S_n>u,S_{n-1}\le u,\hat S_n>xe(u), M_n=Y_n,\hat S_{n-1}>u/2)\\
&\le \PP(Y_n>u/n)\PP(\hat S_{n-1}>xe(u))\\
&\le K (2n)^{\alpha+d\epsilon} (1+\epsilon)^n \F_0(u)\F(xe(u)).
\end{align*}
We are left with
\begin{align*}
 &\PP(S_n>u,\hat S_n>xe(u), M_n=Y_n,S_{n-1}\le u/2,\hat S_{n-1}\le xe(u)/2)\\
&=\int_{0}^{ u/2}\int_{0}^{ \frac{xe(u)}2} \F_0(u-S_{n-1}+x e(u)-\hat S_{n-1}) \de S_{n-1} \de \hat S_{n-1}\\  
&=\int_{0}^{ u/2}\int_{0}^{ \frac{xe(u)}2} \F_0(u+x e(u)) \de S_{n-1} \de \hat S_{n-1}\\
&\quad+\frac 1 \mu \int_{0}^{ u/2}\int_{0}^{ \frac{xe(u)}2} (S_{n-1}+\hat S_{n-1})\F(u+x e(u)-\xi_{u,S_{n-1},\hat S_{n-1}}) \de S_{n-1} \de \hat S_{n-1},
\end{align*}
where $0<\xi_{u,S_{n-1},\hat S_{n-1}}<(u+xe(u))/2$ and hence there exists a constant $C$ such that $\F(u+x e(u) -\xi_{u,S_{n-1},\hat S_{n-1}})\le C\F(u+xe(u))$. It follows by dominated convergence that
\begin{align*}
&\frac 1 \mu \int_{0}^{ u/2}\int_{0}^{ \frac{xe(u)}2} (S_{n-1}+\hat S_{n-1})\F(u+x e(u)-\xi_{u,S_{n-1},\hat S_{n-1}}) \de S_{n-1} \de \hat S_{n-1}\\
&\sim \frac 1 \mu  \mean{S_{n-1}+\hat S_{n-1}} \F(u+xe(u)).
\end{align*}
Note that
\begin{align*}
1-\int_{0}^{ u/2}\int_{0}^{ \frac{xe(u)}2} \de S_{n-1} \de \hat S_{n-1}&
\le \PP\left(S_{n-1}>u/2\right)+\PP\left( \hat S_{n-1}>\frac{xe(u)}2\right)\\
&\le  K (1+\epsilon)^n 2^{n+\epsilon}\F_0(u)\F_0(xe(u)).
\end{align*}
It follows that
\begin{align}
 &\PP(S_n>u,S_{n-1}\le u,\hat S_n>xe(u), M_n=Y_n)\notag\\&=
\PP(S_n>u,\hat S_n>xe(u), M_n=Y_n)+\grossO(\F_0(u)^2) \notag\\&=
\F_0(u+xe(u))+\frac 1 \mu \mean{S_{n-1}+\hat S_{n-1}} \F(u+xe(u)) +o(\F(u)).\label{eq:helpreslast}
\end{align}
Next consider $\{M_n=Y_i\}$ where w.l.o.g we will assume that $i=n-1$.

Then we get with the same method as that leads to \eqref{eq:helpreslast}
\begin{align*}
 &\PP(S_n>u,S_{n-1}<u,\hat S_n>xe(u), M_n=Y_{n-1})\\&=\PP(Y_{n-1}> u-S_{n-2}-Y_n,\hat S_n>xe(u), M_n=Y_{n-1})\\
&\quad -\PP(Y_{n-1}> u-S_{n-2},\hat S_n>xe(u), M_n=Y_{n-1})\\
&=\frac 1\mu \mean{Y_n} \F(u+xe(u))+ o(\F(u)).
\end{align*}

\end{proof}
We also need some properties of the density of $Z$. As in Lemma \ref{assZ} we can get upper bounds such that with dominated convergence we can use a random $n$.
\begin{lemma}\label{lemma:regularityZ}
 Under Assumption \ref{assumption} we get that
\[
 \left .\frac{\de \PP(S_{n-1}\le u,S_n>u,\hat S_n\le x)}{\de x}\right| _{x=yu}\sim \frac{1}{\mu}\F(u+yu).
\]
\end{lemma}
\begin{proof}
 Note that 
\begin{align*}
\PP(\hat S_n>x)= \int_0^x \int_0^u \PP(Y_n>u- S_{n-1},Z_n\le  x-\hat S_{n-1}) \de S_{n-1}\de \hat S_{n-1}   
\end{align*}
 It follows that 
\begin{align}
 & \frac{\de \PP(S_{n-1}\le u,S_n>u,\hat S_n\le x)}{\de x}\notag\\
&= \mu^{-n} \int_{S_{n-1}<u} \int_{u-S_{n-1}}^\infty \int_0^x\int_0^{x-x_1}\cdots\int_{0}^{x-\sum_{i=1}^{n-1} x_i}\notag\\&\quad f(x_1+y_1)\cdots f(x_{n-1}+y_{n-1}) f\left(x-\sum_{i=1}^{n-1} x_i +y_n\right) \de x_1\cdots \de x_{n-1} \de y_1\cdots \de y_{n}\notag\\
&= \mu^{-1}\mean{\F\left(u+x- S_{n-1}- \hat S_{n-1}\right),S_{n-1}<u,\hat S_{n-1}<x/2 }\label{part1}\\
&\quad+ \mu^{-1}\mean{\F\left(u+x- S_{n-1}- \hat S_{n-1}\right),S_{n-1}<u,x/2\le \hat S_{n-1}<x }.\label{part2}
\end{align}
If we choose $x = yu$ for some $y>0$ then we get with dominated convergence that $\eqref{part1}\sim\F(u+x)$ further $\eqref{part1}\le c(y) \F(u+x)$ for some $0<c(y)<\infty$. To find a bound for \eqref{part2} note that the mean over the region $Z_i>x/4n$ and $Z_j>x/4n$ can be bounded by $\F_0(x/4n)^2$ (and since $\mean{X^2} <\infty$ we get that $\F_0(x)^2=o(\F(x))$.
By using the symmetry of the problem in the $Z_i$ we can asymptotically bound the mean of \eqref{part2} by
\begin{align}
&\mean{\F\left(u+x- S_{n-1}- \hat S_{n-1}\right),S_{n-2}> u/4,S_{n-1}<u,x/2\le \hat S_{n-1}<x, \hat S_{n-2}\le x/4}\notag\\
&+\mean{\F\left(u+x- S_{n-1}- \hat S_{n-1}\right),S_{n-2}\le u/4,S_{n-1}<u,x/2\le \hat S_{n-1}<x, \hat S_{n-2}\le x/4}.\notag
\end{align}
If $S_{n-2}>u/4$ then one of the $Y_i$ $i\le n-2$ is bigger then $ u/(4(n-2))$  and we can bound the corresponding mean by $\F_0(x/4)\F_0( u/(4(n-2))$.

It is left to bound the  mean when $S_{n-2}\le u/4$. At first we assume that $S_{n-1}<u/2$ then we can use the same method to bound the mean by
\begin{align*}
 &\mean{\F\left(u+x- S_{n-1}- \hat S_{n-1}\right),S_{n-1}\le u/2,x/2\le \hat S_{n-1}<x, \hat S_{n-2}\le x/4}\\
&\le \mean{\F\left(u/2\right),Z_{n-1}>x/4}=\F\left(u/2\right)\F_0(x/4).
\end{align*}
Finally note that
\begin{align*}
 &\mean{\F\left(u+x- S_{n-1}- \hat S_{n-1}\right),S_{n-2}\le u/4, u/2<S_{n-1}\le u,x/2\le \hat S_{n-1}<x, \hat S_{n-2}\le x/4}\\
 &=\mu^{-1} \mathbb E\Bigg[ \int_{u/2-S_{n-2}}^ {u-S_{n-2}}\int_{x/2-\hat S_{n-2}}^ {x-\hat S_{n-2}} f(z+y)\\& \quad \times \F\left(u+x- S_{n-2}- \hat S_{n-2}-z-y\right)\de z \de y,S_{n-2}\le u/4, \hat S_{n-2}\le x/4\Bigg]\\
 &=\mu^{-1}\mathbb E \Bigg[\int_{0}^ {u/2}\int_{0}^ {x/2} f(u+x- S_{n-2}- \hat S_{n-2}-z-y)\\&\quad\quad\quad\quad\quad\quad\quad\quad\quad\quad\quad\quad \times \F\left(z+y\right)\de z \de y,S_{n-2}\le u/4, \hat S_{n-2}\le x/4\Bigg ]\\
&\lesssim \mu^{-1}f((u+x)/4)\int_{0}^ {u/2}\F\left(z+y\right)\de z \de y.
\end{align*}
The integral in the last equation  is finite since $\mean{X^2}<\infty$.
\end{proof}
Next we consider the derivative of the density of  $\hat S_n$. It follows that
\begin{lemma}\label{lemma:regularityZd} Under Assumption \ref{assumption} we get that
\[
 \left .\frac{\de^2 \PP(S_{n-1}\le u,S_n>u,\hat S_n\le x)}{\de x^2}\right| _{x=yu}\sim \frac{1}{\mu}f(u+yu).
\]
\end{lemma}
\begin{proof}
\begin{align*}
&  \frac{\de^2 \PP(S_{n-1}\le u,S_n>u,\hat S_n\le x)}{\de x^2} \\
&=-\mu^{-1}\mean{f\left(u+x-S_{n-1}-\hat S_{n-1}\right),S_{n-1}\le u, \hat S_{n-1} \le x}\\
&\quad +\mu^{-2}\mean{\int_{0}^{u-S_{n-2}}f \left(x-\hat S_{n-2} +y\right)\F\left(u-S_{n-2}-y\right)\de y ,S_{n-2}\le u, \hat S_{n-2} \le x}\\
&=I_1+I_2.
\end{align*}
We only give a detailed asymptotic analysis  for  $I_2$ (the asymptotic of $I_1$ can be found analogously).  If $S_{n-1}\le u/2$ and $\hat S_{n-2}\le x/2$ then the mean can be asymptotically  bounded by $\F(x/2)\F(u/2)=o(f(x+u))$.
Next we consider the case where $S_{n-1}>u/2$ and only one $Y_i>u/(4n)$.\\ 
At first we assume that  $S_{n-2}\le u/4$ and $u/2< S_{n-1}\le u$. Then
\begin{align*}
&\mu^{-1} \mean{\int_{u/2-S_{n-2}}^{u-S_{n-2}}f \left(x-\hat S_{n-2} +y\right)\F\left(u-S_{n-2}-y\right)\de y ,S_{n-2}\le u/4, \hat S_{n-2} \le x}\\
&=\mu^{-1} \mean{\int_{0}^{u/2}f \left(x+u-\hat S_{n-2} -y-S_{n-2}\right)\F\left(y\right)\de y ,S_{n-2}\le u/4, \hat S_{n-2} \le x}\\
&\sim f(x+u)
\end{align*}
where the last equation follows with dominated convergence. If $Y_{n-1}\le u/(4n)$ then by symmetry it is enough to consider $Y_{n-2}>u/4$. Hence we get
\begin{align*}
&\mu^{-2} \mathbb E\Bigg[ \int^{u/4-S_{n-3}}_0 \int_0^{x-\hat S_{n-3}} \int_{u/2-S_{n-3}-y}^{u-S_{n-3}-y} f(x_{n-2}+y_{n-2})f \left(x-\hat S_{n-3}-z_{n-2} +y\right)\\&\quad\times \F\left(u-S_{n-3}-y_{n-2}-y\right)\de y_{n-2} \de z_{n-2}\de y ,S_{n-3}\le u/4, \hat S_{n-3} \le x\Bigg]\\
&=\mu^{-2} \mathbb E\Bigg[ \int^{u/4-S_{n-3}}_0 \int_0^{x-\hat S_{n-3}} \int_{0}^{u/2} f(x_{n-2}+u-S_{n-3}-y_{n-2}-y)\\&\quad\times f \left(x-\hat S_{n-3}-z_{n-2} +y\right)\F\left(y_{n-2}\right)\de y_{n-2} \de z_{n-2}\de y ,S_{n-3}\le u/4, \hat S_{n-3} \le x\Bigg].
\end{align*}  
For $\hat S_{n-2}\le x/2 $ the above mean is $\grossO(\F(x/2)\F(u/4))=o(f(u+x))$. If $x/2< \hat S_{n-2}<x$ and     $Z_i\le x/4n$  for all but one $i\not= n-2$ the above mean is  $\grossO(\F(x/4x)\F(u/4))$. If more then two $Z_i>x/4n$  $i\not= n-2$  the above mean  is     $\grossO(\F_0(x/4n)^2\F(u/4))$. If $Z_{n-2}>x/4n$ and another $Z_i>x/4n$ then the mean is $\grossO(\F_0(x/4n)\F(3u/4))$. Finally if all $Z_i\le x/4n$. then the above integral is asymptotically the same as $f(u+x)$. Similar we can show that when at least two $Y_i>u/4n$ the integral is asymptotically negligibly and hence
$I_2\sim \mu^ {-1}(n-1)f(u+x)$. With the same method we get that $I_1\sim - \mu^ {-1}nf(u+x)$ and hence the Lemma follows.

Again we get that for $\hat S_{n-1}\le x/2$ that $I_1\sim f(u+x)$ when $x =yu$.

\end{proof}

\appendix
\section{Some auxiliary lemmas}\label{S:Lemmas}
\setcounter{equation}{0}

\begin{lemma}\label{lemma:sinuintegral}
 Assume that for a function $g_u(x)$ such that $\sup_{x,u}|g_u(x)|<\infty$, there exists a function $h(x)$ with $|g_u'(x)|\le h(x)$ for all $u>0$. Then
for every function $a(u)$ we have as $u\to \infty$ 
\[
 \left |\int_{0}^{a(u)} \sin(ux) g_u(x) \de x\right| \le \frac {1} u \int_0^{a(u)} h(x)\de x+o(1).
\]

\end{lemma}
\begin{proof}The Lemma follows by partial integration:
\[
 \int_{0}^{a(u)} \sin(ux) g_u(x) \de x= \frac1 u g_u(0)-\frac{\cos(ua(u))} u g_u(a(u))+\frac 1 u \int_{0}^{a(u)} \cos(ux) g_u'(x) \de x.
\]
\end{proof}

\begin{lemma} \label{lemma:smoothing}
 Assume that $E$ is non lattice and that $\mean{E^2}<\infty$  and $h(z)=  \sum_{i=1}^{N(z)} E_i - \lambda z \mean{E}$ and $N_u$ 
a normal random variable with mean zero and variance $\sigma^2\sim e(u)^{-k}$ for some $c>0$, $k>0$. Then the random variable $N_u+h(xe(u))/\sqrt{xe(u)}$ has a differentiable
density $f_u$.  Further, if  $a,b$ are arbitrary but fixed, it holds uniformly for $w$ and $0<a<x<b<\infty$
that
\[
 \lim_{u\to\infty} f_u(w)= \frac {\exp\left(-\frac {w^2}{\lambda \mean{E^2}}\right)} {\sqrt{2\pi\lambda \mean{E^2}} },\quad \lim_{u\to\infty} f_u'(w)= \frac {-2w\exp\left(-\frac {w^2}{\lambda \mean{E^2}}\right)} {\sqrt{2\pi(\lambda \mean{E^2})^3}}.
\]
If further $k\ge 4$ then
\begin{align*}
\left|\PP\left(N_u +\frac{h(xe(u))}{\sqrt{x e(u)}} >w\right)- \PP\left(\frac{h(xe(u))}{\sqrt{x e(u)}} >w\right)\right| =o\left(\frac 1{e(u)}\right).
\end{align*}
\end{lemma}
 \begin{proof}
Denote with $\chi_E(s)$ the characteristic function of $E$ and with $\sigma^2=\lambda \mean{E^2}$. Note that the Fourier transform of $f'_u(w)-f'_{N(0,\sigma^2)}(w)$ is
\[
  i s \left(\e^{-\frac {\sigma^2_u s^2}2} \e^{\lambda x e(u) \left(\chi_E\left(s/\sqrt{xe(u)}\right)-1\right) - i \sqrt{ xe(u)}\lambda\mean{E}s }-  \e^{-\frac {\sigma^2 s^2}2} \right)
\]
 and hence 
\begin{align*}
 &|f'_u(w)-f'_{N(0,\sigma^2)}(w)|\\&\le \int_{-\infty}^\infty |s|\left|\e^{-\frac {\sigma^2_u s^2}2} \e^{\lambda x e(u) \left(\chi_E\left(s/\sqrt{x e(u)}\right)-1\right) - i  \sqrt{x e(u)}\lambda\mean{E} s  }-  \e^{-\frac {\sigma^2 s^2}2} \right| \de s.
\end{align*}
Choose an $\epsilon>0$  such that for $|x|\le \epsilon$,  $\Ree(\chi_E''(x))$ is bounded away from $0$.
Since there exists a $\delta>0$ such that for all $|s|>\epsilon$, $\Ree(1-\chi_E(s))\ge \delta$ ($E$ is non lattice).
\begin{align*}
& \int_{\epsilon \sqrt{xe(u)}}^\infty |s|\left|\e^{-\frac {\sigma^2_u s^2}2} \e^{\lambda x e(u) \left(\chi_E\left(s/\sqrt{xe(u)}\right)-1\right) - i \sqrt{x e(u)} \lambda\mean{E} s }-  \e^{-\frac {\sigma^2 s^2}2} \right| \de s\\
&\le \e^{-\delta \lambda x e(u)} \int_{\epsilon \sqrt{xe(u)}}^\infty s \e^{-\frac {\sigma^2_u s^2}2}\de s + \int_{\epsilon{\sqrt{xe (u)}}} ^\infty s\e^{-s^2} \de s\\
&\le \frac 1 {\sigma_u^2} \e^{-\delta \lambda x e(u)} \int_0^\infty s \e^{-\frac { s^2}2}\de s + \int_{\epsilon \sqrt{xe (u)}} ^\infty s\e^{-s^2} \de s \to 0
\end{align*}
as $u\to \infty$. With the same arguments
\[
\int^{-\epsilon\sqrt{xe(u)}}_{-\infty} |s|\left|\e^{-\frac {\sigma^2_u s^2}2} \e^{\lambda x e(u) \left(\chi_E\left(s/\sqrt{xe(u)}\right)-1\right) - i \sqrt{x e(u)} \lambda\mean{E} s}-  \e^{-\frac {\sigma^2 s^2}2} \right| \de s\to 0.
\]
Further for  a $\xi_{u,s}$ bounded away from $0$ and $\xi_{u,s}\to \mean{E^2}$ for fixed $s$ as $u\to\infty$
\begin{align*}
&\int^{\epsilon\sqrt{x e(u)}}_{-\epsilon \sqrt{xe(u)}} |s|\left|\e^{-\frac {\sigma^2_u s^2}2} \e^{\lambda x e(u) \left(\chi_E\left(s/\sqrt{xe(u)}\right)-1\right) - i \sqrt{x e(u)}\lambda\mean{E}s }-  \e^{-\frac {\sigma^2 s^2}2} \right| \de s\\
&=\int^{\epsilon\sqrt{xe(u)}}_{-\epsilon\sqrt{xe(u)}} |s|\left|\e^{-\frac {\sigma^2_u s^2}2} \e^{-\lambda \xi_{s,u} s^2/{2} }-  \e^{-\frac {\sigma^2 s^2}2} \right| \de s.
\end{align*}
By dominated convergence we get that the last integral tends to $0$ as $u\to\infty$.\\
Since the estimate of   $|f_u(w)-f_{N(0,\sigma^2)}(w)|$ works with exactly the same arguments we leave it to the reader.\\
Denote with $\chi_u$ is characteristic function of $h(x e(u))/\sqrt{xe(u)}$.
Since we can find an $m$ such that $f_u(w)\le m$ for all $w$ and $u$ we get by Lemma XVI.4  2 of \cite{Feller:1971}
that 
\begin{align*}
  &\left|\PP\left(N_u +\frac{h(xe(u))}{\sqrt{x e(u)}} >w\right)- \PP\left(\frac{h(xe(u))}{\sqrt{x e(u)}} >w\right)\right|\\ &\le \frac 1 \pi \int_{-T}^T \left |\frac {\left(1-\e^{-\frac {s^2\sigma_u^2}2}\right)\chi_u(s)} {s}\right| \de s +\frac {24 m}{\pi T}\le \frac 1 \pi \int_{-T}^T \left |\frac {s\sigma^2_u} {2}\right| \de s +\frac {24 m}{\pi T}=\frac {\sigma^2_u T^2}{2 \pi} +\frac {24 m}{\pi T}.
\end{align*}
For $T=e(u)^{1 +\epsilon}$ for some $0<\epsilon<1/2$ and $\sigma_u^2\le e(u)^{-4}$ we get that
\[
 \left|\PP\left(N_u +\frac{h(xe(u))}{\sqrt{x e(u)}} >w\right)- \PP\left(\frac{h(xe(u))}{\sqrt{x e(u)}} >w\right)\right|\le \frac {e(u)^{-2+2\epsilon}}{2 \pi} +\frac {24 m}{\pi e(u)^{1+\epsilon}}.
\]
 
 \end{proof}
\begin{lemma} \label{lemma:boundderivative} Under Assumption \ref{assumption}, let
 $h(z)=  \sum_{i=1}^{N(z)} E_i - \lambda z \mean{E_i}$ and let $N_u$ be
a normal random variable with mean zero and variance $\sigma^2\sim e(u)^{-k}$ 
for some $c>0$, $k>0$. Then the random variable $N_u+h(xe(u))/\sqrt{xe(u)}$ has a differentiable
density $f_u$. Further, if $a,b$ are arbitrary but fixed it holds
uniformly for $w$ and $0<a<x<b<\infty$ that \[
 w^3 f_u'(w)\text{ and }w^2 f_u(w)
\]
are  bounded for $w>w_0>0$ and all $u>u_0$ where $u_0$ is choosen such that $xe(u)>1$.
\end{lemma}
\begin{proof}
Denote with ${\hat F_E}(s)=\mean{\e^{-sE_i}}$ and with
\[
 A(s,u)= \lambda \sqrt{x e(u) }{\hat F_E}'\left(s/\sqrt{xe(u)}\right)+\sqrt{ xe(u)}\lambda\mean{E}.
\]

 Note that the (bilateral) Laplace transform of transform of  $w^3 f_u'(w)$ is given by
\begin{align*}
\laplace{w^3f_u'}(s)& =  \frac{\de}{\de s^3}\left( s \e^{ \frac {\sigma^2_u s^2}2} \e^{\lambda x e(u) \left({\hat F_E}\left(s/\sqrt{xe(u)}\right)-1\right) +\sqrt{ xe(u)}\lambda\mean{E}s } \right) \\
&=\e^{ \frac {\sigma^2_u s^2}2} \e^{\lambda x e(u) \left({\hat F_E}\left(s/\sqrt{xe(u)}\right)-1\right) +\sqrt{ xe(u)}\lambda\mean{E}s } \\
&\quad \times \Bigg\{\left(s A(s,u)+\sigma_u^2 s^2 \right)^2\left(
1 + s A(s,u)+\sigma_u^2 s^2 \right)\\
&\quad+\left(s A(s,u)+\sigma_u^2 s^2 \right) \left(A(s,u) +2 \sigma_u^2s+ \lambda s{\hat F_E}''\left(s/\sqrt{xe(u)}\right)\right)\\
&\quad +\left(
1 + 2s A(s,u)+2\sigma_u^2 s^2 \right)\left(A(s,u) +2 \sigma_u^2s+ \lambda s{\hat F_E}''\left(s/\sqrt{xe(u)}\right)\right)\\
&\quad +\left(\lambda {\hat F_E}''\left(s/\sqrt{xe(u)}\right) +2 \sigma_u^2+ \lambda\frac{s}{\sqrt{xe(u)}} {\hat F_E}'''\left(s/\sqrt{xe(u)}\right)\right)\Bigg\}.
\end{align*}
Note that for every $w>w_0$ and $0<\epsilon<1$
\[
 w^3f_u'(w)=\frac 1{2\pi} \int_{-\infty}^\infty  \e^{ w(\epsilon/w + \I s)}\laplace{w^3f_u'}(\epsilon/w+\I s) \de s.
\]
Since 
\[
 A(s,u)=\lambda \sqrt{x e(u) }{\hat F_E}'\left(s/\sqrt{xe(u)}\right)+\sqrt{ xe(u)}\lambda\mean{E}=\lambda s \hat F_E''\left(\xi_{s,u}\right)
\]
$|{\hat F_E}''\left(s\right)|\le \mean{E^2}$ and
$
 s {\hat F_E}'''\left(s \right)
$
is bounded (see Lemma \ref{lemma:boundedlaplace} below) for $|s|<1$, we get that for $|s|<1$ the term in the curly brackets can be bounded  by a polynomial in $|s|$. Hence the Lemma follows analogously to the proof  of \ref{lemma:smoothing}.
\end{proof}

\begin{lemma} \label{lemma:boundedlaplace}
  Under Assumption \ref{assumption}  $s\hat F_E'''(s)$ is uniformly bounded for $s\to0 $
\end{lemma}
\begin{proof}
 Note that 
\[
 E\stackrel{d}{=} X+\sum_{i=1}^{N(X)}E_i
\]
and hence
\[
 \hat F_E(s)=\mean{\e^{-sX+\lambda X (\hat F_E(s)-1)}}=\hat F(s-\lambda (\hat F_E(s)-1))
\]
Since for $\Ree(s)> 0$, $|\hat F_E(s)|< 1$ and hence $\Ree(s-\lambda (\hat F_E(s)-1))>0$ hence the above formula is valid for all $\Ree(s)>0$. Hence both sides are infinitely often differentiable for all $\Ree(s)>0$ and we have
\begin{align*}
\hat F_E'(s)&=\frac{\hat F'\left(s-\lambda (\hat F_E(s)-1)\right)}{1+\lambda\hat F'\left(s-\lambda (\hat F_E(s)-1)\right) },\\
\hat F_E''(s)&=\frac{\hat F''\left(s-\lambda (\hat F_E(s)-1)\right)\left(1-\lambda \hat F_E'(s)\right)^2} {1+\lambda\hat F'\left(s-\lambda (\hat F_E(s)-1)\right) },\displaybreak[0]\\
\hat F_E'''(s)&=\frac{\hat F'''\left(s-\lambda (\hat F_E(s)-1)\right)\left(1-\lambda \hat F_E'(s)\right)^3}{1+\lambda\hat F'\left(s-\lambda (\hat F_E(s)-1)\right) } \\
&\quad - \frac{ 2\lambda  \hat F_E''(s)F''\left(s-\lambda (\hat F_E(s)-1)\right)\left(1-\lambda \hat F_E'(s)\right)}{1+\lambda\hat F'\left(s-\lambda (\hat F_E(s)-1)\right) }\\
&\quad-\frac{\lambda  \hat F_E''(s)F''\left(s-\lambda (\hat F_E(s)-1)\right)\left(1-\lambda \hat F_E'(s)\right)}{1+\lambda\hat F'\left(s-\lambda (\hat F_E(s)-1)\right) }.
\end{align*}
Since $\lambda \mean{X}<1$ we have that  \[\sup_{\Ree(s)\ge0}\left|\lambda\hat F'\left(s-\lambda (\hat F_E(s)-1)\right)\right|<1\]
and hence $\hat F_E'(s)$ is bounded for all $\Ree(s)>0$ and since  $\mean{X^2}<\infty$ also $\hat F_E''(s)$ is bounded. Finally we get that $s\hat F'''_E(s)$ is bounded
Since $s\hat F'''(s)$ is bounded and
\[
 s-\lambda(\F_E(s)-1)=s-\lambda (\hat F_E(s)-1)=s(1-\lambda\hat F'_E(s))+\frac{s^2}2 \hat F''_E(\xi_s)=\grossO(s).
\]

\end{proof}

\begin{lemma}\label{lemma:bound} Let $E_i$ be iid with $\mean{E}<\infty$ and $N(t)$ a Poisson process with intensity $\lambda$ independent of the $E_i$. Then there exists constants $C_1$, $C_2$ and $\delta>0$ such that uniformly for $x>\epsilon t $ 
\[
\PP\left( \left|\sum_{1=1}^{N(t)}E_i -\lambda t \mean{E}\right| >x\right) \le C_1 t \PP(E>x)+ e^{-\delta\left(x-\frac {\epsilon} 2 t\right)}.
\]
\end{lemma}
\begin{proof}
 In \cite{KM:97} it is proved that
\[
\PP\left( \sum_{1=1}^{N(t)}E_i -\lambda t \mean{E} >x\right) \le C_1 t \PP(E>x)
\]
uniformly for $x>\epsilon t$. We can find a $\delta>0$ such that for all $t>0$
\[
\mean{\exp\left(-\delta \left(\sum_{1=1}^{N(t)}E_i -t\left(\lambda  \mean{E}+\frac \epsilon 2\right) \right)  \right)}\le 1 .
\]
The Lemma follows  by the Chernoff bound.
\end{proof}

We often used the following Lemma without further mentioning. Since we don't have a reference by hand we give for completeness  a proof .
\begin{lemma}
 Let $L(x)$ be slowly varying and 
\[
 \int_{0}^\infty \frac 1x L(x)\de x <\infty,
\]
then $\lim_{x\to\infty} L(x)=0$
\end{lemma}
\begin{proof}
 Assume that the Lemma is not true, i.e. there exists a series of points  $x_n$ with $x_n\to \infty$ and  $L(x_n)>\delta$. W.l.o.g. assume that
\[
\inf_{1\le t\le 2} \frac{L(t x_n)}{L(x_n)}>1/2. 
\]
Then
\[
\int_{x_n}^{2x_n} \frac 1 x L(x) \de x\ge \frac \delta 2 \int_{x_n}^{2x_n} \frac 1 x  \de x=\frac {\delta \log (2)} 2, 
\]
which contradicts the conditions of the Lemma.

\end{proof}
\paragraph{Acknowledgment}
We would like to thank Jens Ledet Jensen for pointing out the proof of Lemma \ref{lemma:sinuintegral}.


\begin{thebibliography}{10}

\bibitem{AlbrecherHippKortschak:09}
H.~Albrecher, C.~Hipp, and D.~Kortschak.
\newblock Higher-order expansions for compound distributions and ruin
  probabilities with subexponential claims.
\newblock {\em Scand. Actuar. J.}, (2):105--135, 2010.

\bibitem{AsmussenAlbrecher:2010}
S.~Asmussen and H.~Albrecher.
\newblock {\em Ruin Probabilities}.
\newblock Advanced Series on Statistical Science \& Applied Probability, 14.
  World Scientific Publishing Co. Pte. Ltd., Hackensack, NJ, second edition,
  2010.
  
  \bibitem{AsmussenHojgaard:1996} S.~Asmussen and B.~H\o jgaard (1996)
  Finite horizon ruin probabilities for Markov-modulated risk processes with heavy tails.
  {\em Th.\ Random Processes} {\bf 2}, 96--107.

\bibitem{AsmussenKlueppelberg:1996}
S.~Asmussen and C.~Kl{\"u}ppelberg.
\newblock Large deviations results for subexponential tails, with applications
  to insurance risk.
\newblock {\em Stochastic Process. Appl.}, 64(1):103--125, 1996.

\bibitem{BaltrunasOmey:1998}
A.~Baltr{\=u}nas and E.~Omey.
\newblock The rate of convergence for subexponential distributions.
\newblock {\em Liet. Mat. Rink.}, 38(1):1--18, 1998.

\bibitem{BarbeMcCormick:2004}
P.~Barbe and W.~P. McCormick.
\newblock {\em Asymptotic Expansions for Infinite Weighted Convolutions of
  Heavy Tail Distributions and Applications}.
\newblock American Mathematical Society, 2009.

\bibitem{BarbeMcCormickZhanga:2006}
P.~Barbe, W.~P. McCormick, and C.~Zhang.
\newblock Asymptotic expansions for distributions of compound sums of random
  variables with rapidly varying subexponential distribution.
\newblock {\em J. Appl. Probab.}, 44(3):670--684, 2007.

\bibitem{BinChongfengWeidong:2011}
T.~Bin, W.~Chongfeng, and X.~Weidong.
\newblock Risk concentration of aggregated dependent risks: The second-order
  properties.
\newblock {\em Insurance: Mathematics and Economics}, (0):--, 2011.

\bibitem{BGT:87}
N.~H. Bingham, C.~M. Goldie, and J.~L. Teugels.
\newblock {\em Regular Variation}, volume~27 of {\em Encyclopedia of
  Mathematics and its Applications}.
\newblock Cambridge University Press, Cambridge, 1989.

\bibitem{Borovkov:2009}
A.~A. Borovkov.
\newblock Insurance with borrowing: first- and second-order approximations.
\newblock {\em Adv. in Appl. Probab.}, 41(4):1141--1160, 2009.

\bibitem{DegenLambriggerSegers:2010}
M.~Degen, D.~D. Lambrigger, and J.~Segers.
\newblock Risk concentration and diversification: second-order properties.
\newblock {\em Insurance Math. Econom.}, 46(3):541--546, 2010.

\bibitem{EMK:97}
P.~Embrechts, C.~Kl{\"u}ppelberg, and T.~Mikosch.
\newblock {\em Modelling Etremal Events for Insurance and Finance}, volume~33
  of {\em Applications of Mathematics (New York)}.
\newblock Springer-Verlag, Berlin, 1997.

\bibitem{Feller:1971}
W.~Feller.
\newblock {\em An Introduction to Probability Theory and its Applications, II
  (2nd ed.)}.
\newblock John Wiley \& Sons Inc., New York, 1971.

\bibitem{GilPelaez:1951}
J.~Gil-Pelaez.
\newblock Note on the inversion theorem.
\newblock {\em Biometrika}, 38:481--482, 1951.


\bibitem{KKM:04}
C.~Kl{\"u}ppelberg, A. Kyprianou and R.\ Maller.
\newblock Ruin probabilities and overshoots for general L\'evy insurance risk models.
\newblock {\em Ann. Appl. Probab.}, 14:1766--1801, 2004

\bibitem{KM:97}
C.~Kl{\"u}ppelberg and T.~Mikosch.
\newblock Large deviations of heavy-tailed random sums with applications in
  insurance and finance.
\newblock {\em J. Appl. Probab.}, 34(2):293--308, 1997.

\bibitem{K:2011}
D.~Kortschak.
\newblock Second order tail asymptotics for the sum of dependent,
  tail-independent regularly varying risks.
\newblock {\em Extremes}, pages 1--36, 2011.
\newblock 10.1007/s10687-011-0142-x.

\bibitem{OmeyWillekens:1986}
E.~Omey and E.~Willekens.
\newblock Second order behaviour of the tail of a subordinated probability
  distribution.
\newblock {\em Stochastic Process. Appl.}, 21(2):339--353, 1986.

\bibitem{OmeyWillekens:1987}
E.~Omey and E.~Willekens.
\newblock Second-order behaviour of distributions subordinate to a distribution
  with finite mean.
\newblock {\em Comm. Statist. Stochastic Models}, 3(3):311--342, 1987.

\bibitem{Resnick:86}
S.~I. Resnick.
\newblock Point processes, regular variation and weak convergence.
\newblock {\em Adv. in Appl. Probab.}, 18(1):66--138, 1986.

\bibitem{Wendel:1961}
J.~G. Wendel.
\newblock The non-absolute convergence of {G}il-{P}elaez' inversion integral.
\newblock {\em Ann. Math. Statist.}, 32:338--339, 1961.

\end{thebibliography}

\end{document}